\documentclass[12pt]{amsart}

\usepackage{graphicx}

\usepackage{amsmath,amssymb,amsfonts,amsthm,amsopn}

\newtheorem{theorem}{Theorem}[section]
\newtheorem{lemma}[theorem]{Lemma}
\newtheorem{corollary}[theorem]{Corollary}
\newtheorem{proposition}[theorem]{Proposition}

\newtheorem{remark}[theorem]{Remark}

\numberwithin{equation}{section}

\newcommand{\tfa}{time-frequency analysis}

\newcommand{\ft}{Fourier transform}

\newcommand{\beqa}{\begin{eqnarray*}}
\newcommand{\eeqa}{\end{eqnarray*}}

\newcommand{\field}[1]{\mathbb{#1}}
\newcommand{\bR}{\field{R}}        
\newcommand{\bC}{\field{C}}        
\def\rn{\mathbb{R}^d}

\newcommand{\rnn}{\mathbb{R}^{2d}}

 \def\cF{\mathcal{F}}              
 \def\cS{\mathcal{S}}
 \def\cD{\mathcal{D}}

\def\rd{\bR^d}

\def\rdd{{\bR^{2d}}}

\def\intrd{\int_{\rd}}
\def\intrdd{\int_{\rdd}}

\def\R{\right)}

\def\<{\left<}
\def\>{\right>}

\def\mv1{M_v^1}

\def\o{\omega}

\def\R{\mathbb{R}}
\def\Ren{\mathbb{R}^d}

\def\Fur{\mathcal{F}}

\def\Sn2{S_{2}(L^{2}(\Ren))}
\def\S1{S_{1}(L^{2}(\Ren))}
\def\sig00{\sigma_{0,0}}

\newcommand{\fpq}{W(\mathcal{F}L^p,L^q)}

\begin{document}

\title[Interferences in the Born--Jordan distribution]{On the reduction of the interferences in the Born--Jordan distribution}

\author{Elena Cordero}
\address{Dipartimento di Matematica,
Universit\`a di Torino, Dipartimento di Matematica, via Carlo Alberto 10, 10123 Torino, Italy}
\curraddr{}
\email{elena.cordero@unito.it}
\thanks{}

\author{Maurice de Gosson}
\address{University of Vienna, Faculty of Mathematics, Oskar-Morgenstern-Platz 1
A-1090 Wien, Austria}
\curraddr{}
\email{maurice.de.gosson@univie.ac.at}
\thanks{}

\author{Fabio Nicola}
\address{Dipartimento di Scienze Matematiche,
Politecnico di Torino, corso Duca degli Abruzzi 24, 10129 Torino,
Italy}
\curraddr{}
\email{fabio.nicola@polito.it}
\thanks{}


\date{}

\dedicatory{}

\begin{abstract}
One of the most popular time-frequency representation is certainly the Wigner distribution. To reduce the interferences coming from its quadratic nature, several related distributions have been proposed, among which the so-called Born-Jordan distribution. It is well known that in the Born-Jordan distribution the ghost frequencies are in fact damped quite well, and the noise is in general reduced. However, the horizontal and vertical directions escape from this general smoothing effect, so that the interferences arranged along these directions are in general kept. Whereas these features are graphically evident on examples and heuristically well understood in the engineering community, there is not at present a mathematical explanation of these phenomena, valid for general signals in $L^2$ and, more in general, in the space $\cS'$ of temperate distributions. In the present note we provide such a rigorous study using the notion of wave-front set of a distribution. We use techniques from Time-frequency Analysis, such as the modulation and Wiener amalgam spaces, and also results of microlocal regularity of linear partial differential operators.
\end{abstract}

\maketitle

\section{Introduction}
The representation of signals in the time-frequency plane is a fascinating theme involving several mathematical subtilities, mostly related to some form of the uncertainty principle. The most popular time-frequency distribution is without any doubt the Wigner distribution, defined by 
\begin{equation}\label{wigner}
W f(x,\omega)=\int_{\rd} f\big(x+\frac{y}{2}\big) \overline{f\big(x-\frac{y}{2}\big)}e^{-2\pi i y\omega}\, dy\qquad x,\omega\in\rd
\end{equation}
where the signal $f$ is, say, in the space $\cS'(\rd)$ of temperate distributions in $\rd$. The quadratic nature of this representation, however, causes the appearance of interferences between several components of the signal. To damp this undesirable effect numerous {\it Reduced Interference Distributions} have been proposed; see \cite{Cohen2,auger} for a detailed account. Here we will focus on the Born-Jordan distribution, first introduced in \cite{Cohen1}, and defined by
\begin{equation}\label{bj}
Q f= Wf \ast \Theta_\sigma
\end{equation}
where $\Theta$ is Cohen's kernel function,
given by
\begin{equation}\label{sincxp}
\Theta(\zeta_1,\zeta_2)=\frac{\sin(\pi \zeta_1\zeta_2)}{\pi \zeta_1\zeta_2},\quad \zeta=(\zeta_1,\zeta_2)\in\rdd
\end{equation}
 ($\zeta_1\zeta_2= \zeta_1\cdot \zeta_2$ being the scalar product in $\rd$ and $\Theta(\zeta)=1$ for $\zeta=0$), and $\Theta_\sigma(\zeta)=\Theta_\sigma(\zeta_1,\zeta_2)=\Fur a(\zeta_2,-\zeta_1)$, with $\zeta=(\zeta_1,\zeta_2)$, is the symplectic Fourier
transform of $\Theta$, see
\cite{bogetal, Cohen1,Cohen2,Cohenbook,cgn0,TRANSAM, golu1,auger} and
the references therein.\par
To motivate our results and for the benefit of the reader, we now compare (graphically) the features of the Wigner and Born-Jordan distributions of some signals. The following remarks are well known and we refer to \cite{bo3,bogetal,choi,Cohen2,hla,jeong,loughlin1,loughlin2,turunen,zhao} and especially to \cite{auger} for more details; we also refer to the comprehensive list of references at the end of \cite[Chapter 5]{auger} for the relevant engineering literature about the geometry of interferences and kernel design.
\subsection{Graphical comparisons}
Graphical examples both for test signals and real-world signals show that the Born-Jordan distribution, in comparison with the Wigner one, enjoys the following features (in dimension $d=1$):
\begin{itemize}
\item[a)] the so-called ``ghost frequencies'', arising from the interferences of the several components which do not share the same time or frequency localization, are damped very well.
\item[b)] The interferences arranged along the horizontal and vertical direction are substantially kept.
\item[c)] The noise is, on the whole, reduced.
\end{itemize}
\par\medskip
These facts can be interpreted (still in dimension $d=1$) in terms of the following 
\par\medskip\noindent{\it {\bf Principle.} Compared with the Wigner distribution, the Born-Jordan distribution exhibits a general smoothing effect, which however does not involve the horizontal and vertical directions.}\par\medskip 
The persistence of interference terms in ``vertical" and ``horizontal" directions is closely related to the fact that the Born-Jordan distribution preserves the so-called ``marginal distributions". As is well-known, the requirement of preserving marginals in Cohen's class implies $\Theta(\zeta_1,\zeta_2)$ to be such that $\Theta(\zeta_1,0) = 1$ and $\Theta(0,\zeta_2) = 1$, a property which is satisfied by the Born-Jordan kernel. As a result, since two synchronous components with different frequencies exhibit a beating effect ending up with a modulated envelope in time, this must have a signature in the time-frequency plane, and this is exactly given by the ``vertical" cross-terms in between the two considered components. By symmetry, the same applies as well ``horizontally" to two components that are disjoint in time but in the same frequency band.\par

The Figures \ref{figura1}--\ref{figura3} illustrate the above principle.
In Figure \ref{figura1} the Wigner and Born-Jordan distribution of the sum of $4$ Gabor atoms is displayed; notice the presence of $6$ interferences (two of which superimposed at the center). In general any two components generate an interference centered in the middle point and arranged along the line joining those components. Notice the disappearance of the ``diagonal'' interferences in the Born-Jordan distribution.\par
Figure \ref{figura2} displays the Wigner and Born-Jordan distribution of a chirp signal embedded in a white Gaussian noise. \par
Figure \ref{figura3} shows the Wigner and Born-Jordan distribution of a bat sonar signal, which presents inner interferences due to the curve nature of the component. 
\begin{figure}
\centering
\includegraphics[height=5cm,width=5cm]{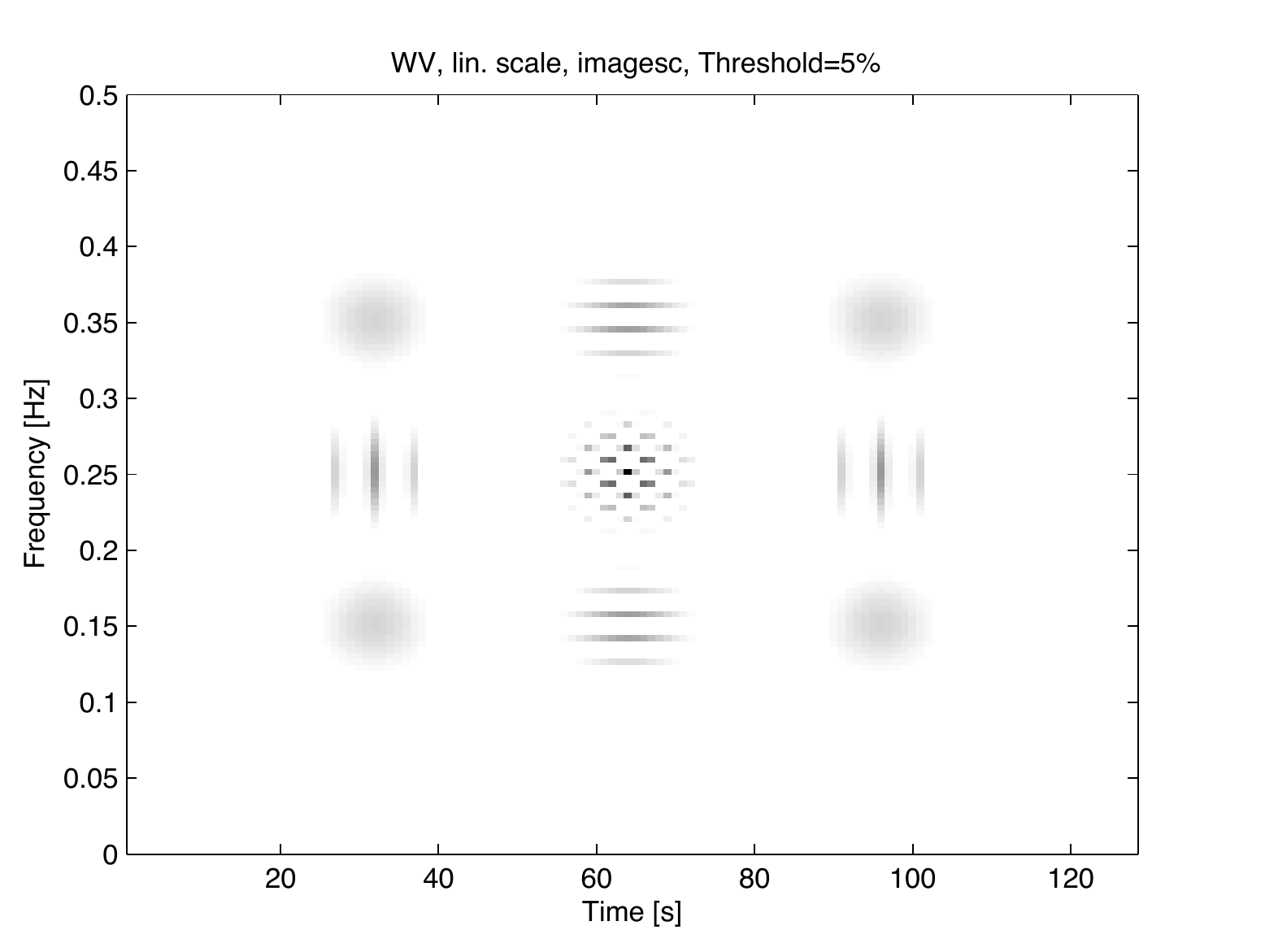}\quad\quad\includegraphics[height=5cm,width=5cm]{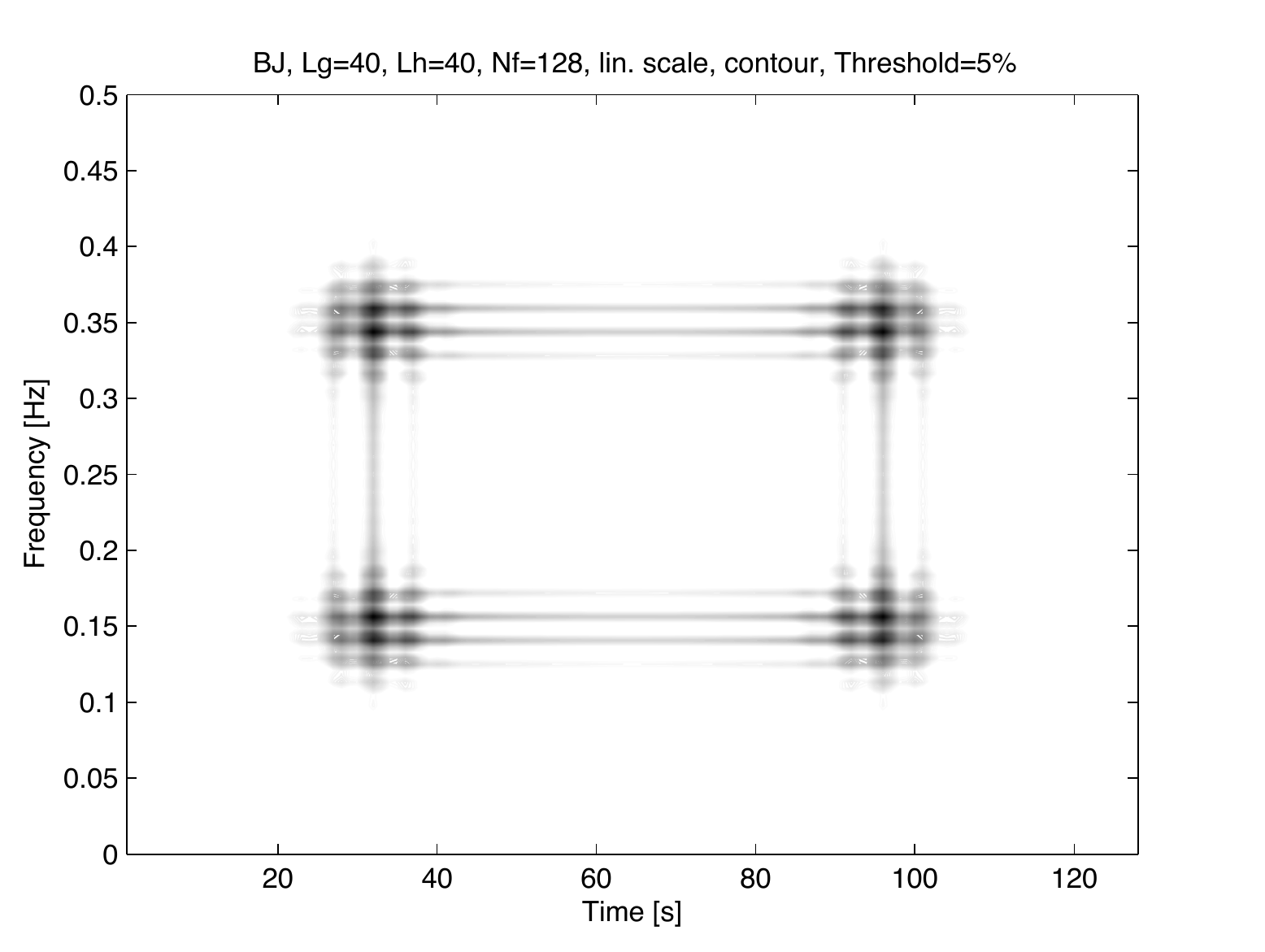}
\caption{Wigner and Born-Jordan distribution of the sum of $4$ Gabor atoms.}
\label{figura1}
\end{figure}

\begin{figure}
\centering
\ \includegraphics[height=5cm,width=5cm]{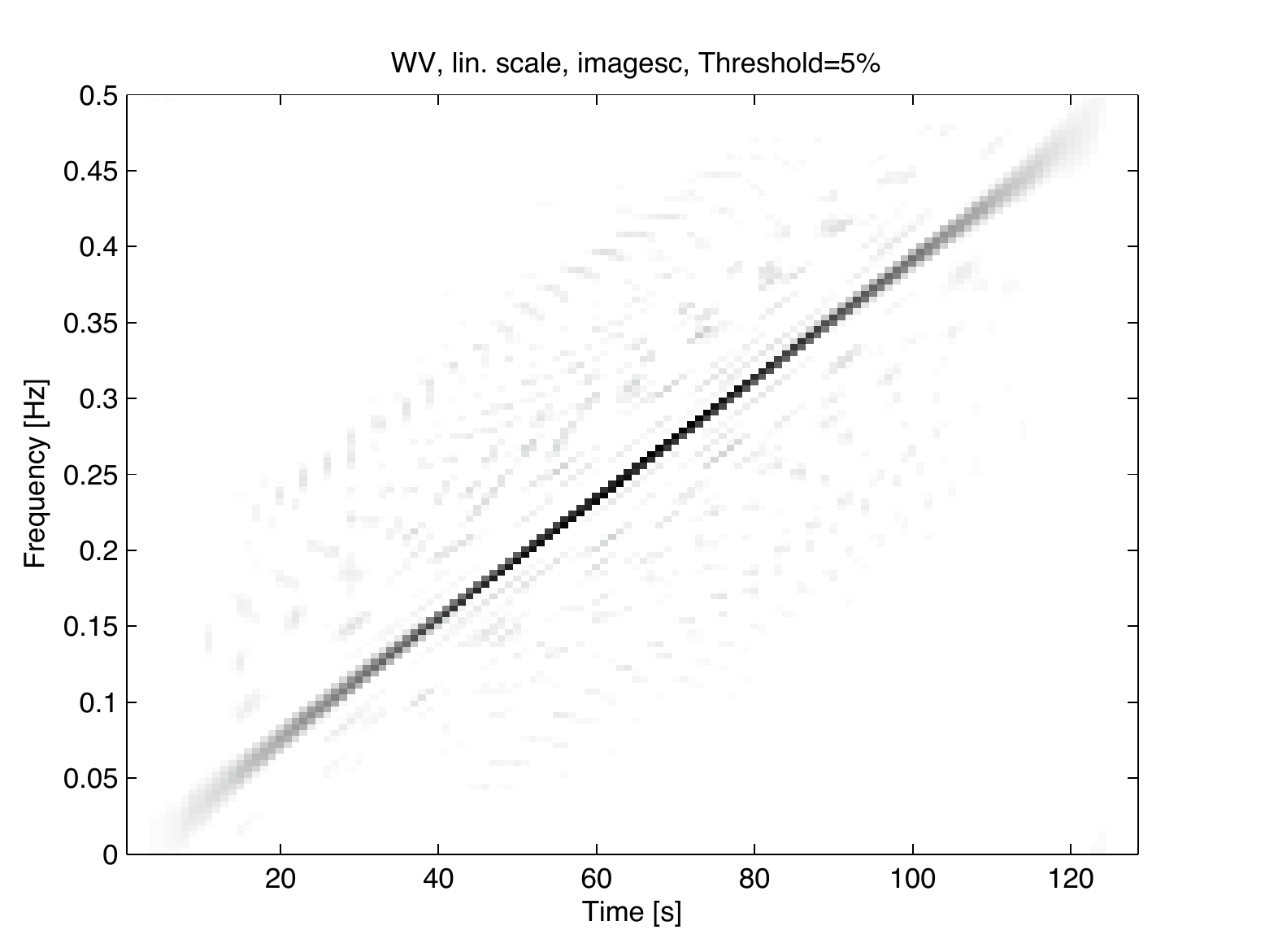}\quad\quad\includegraphics[height=5cm,width=5cm]{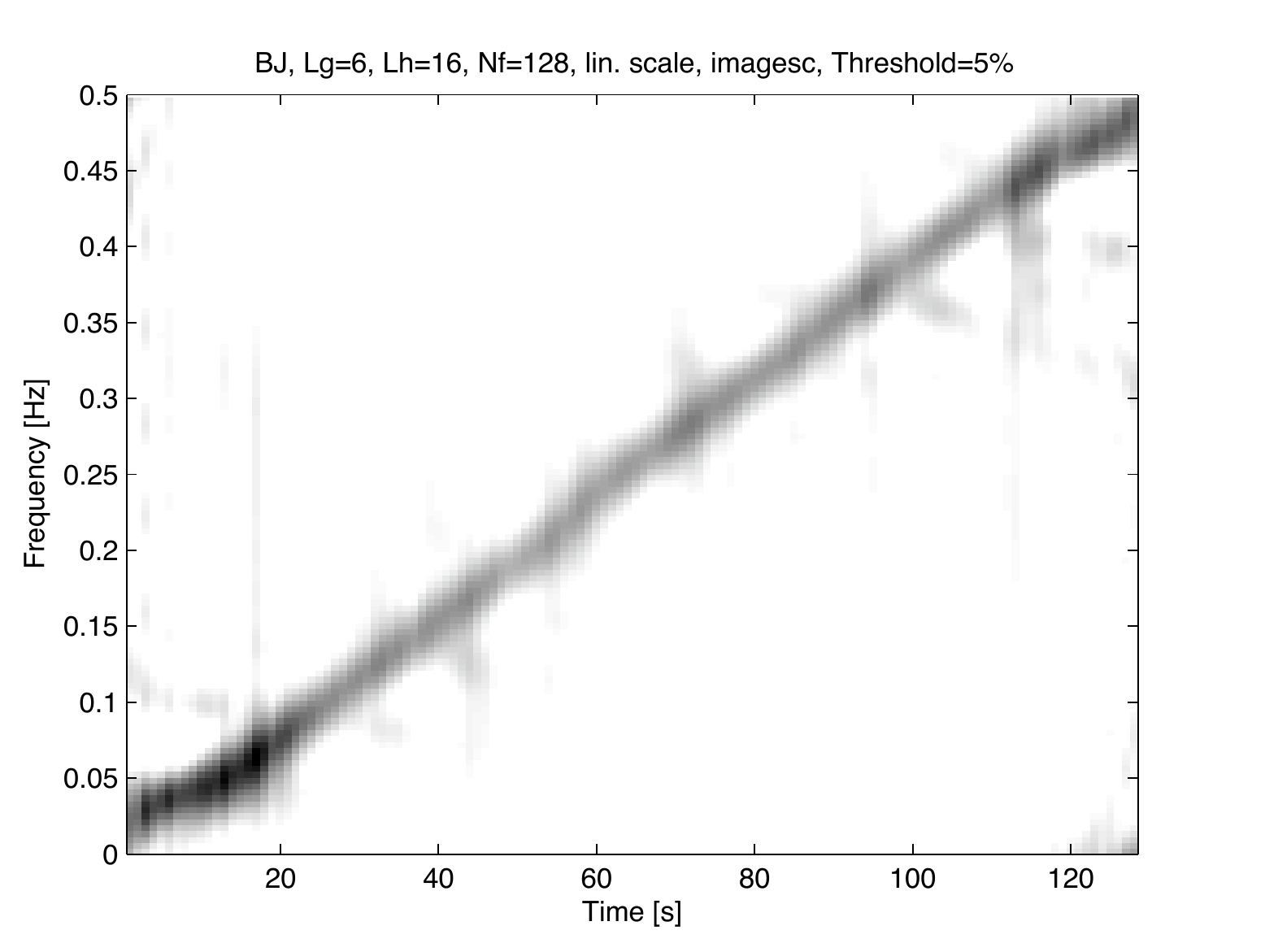}
\caption{Wigner and Born-Jordan distribution of a linear chirp embedded in a white Gaussian noise.}
\label{figura2}
\end{figure}
\begin{figure}
\centering
\ \includegraphics[height=5cm,width=5cm]{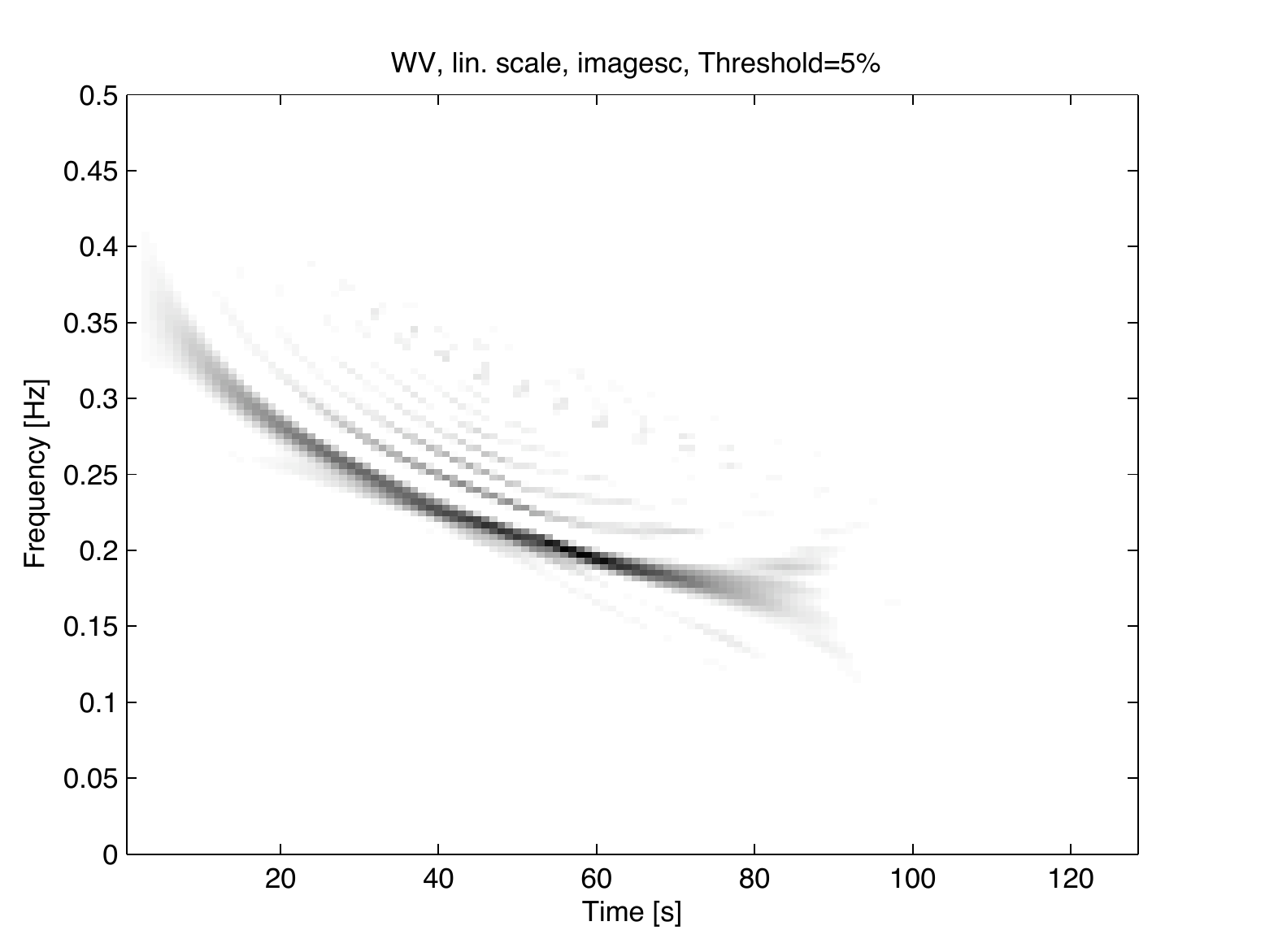}\quad\quad\includegraphics[height=5cm,width=5cm]{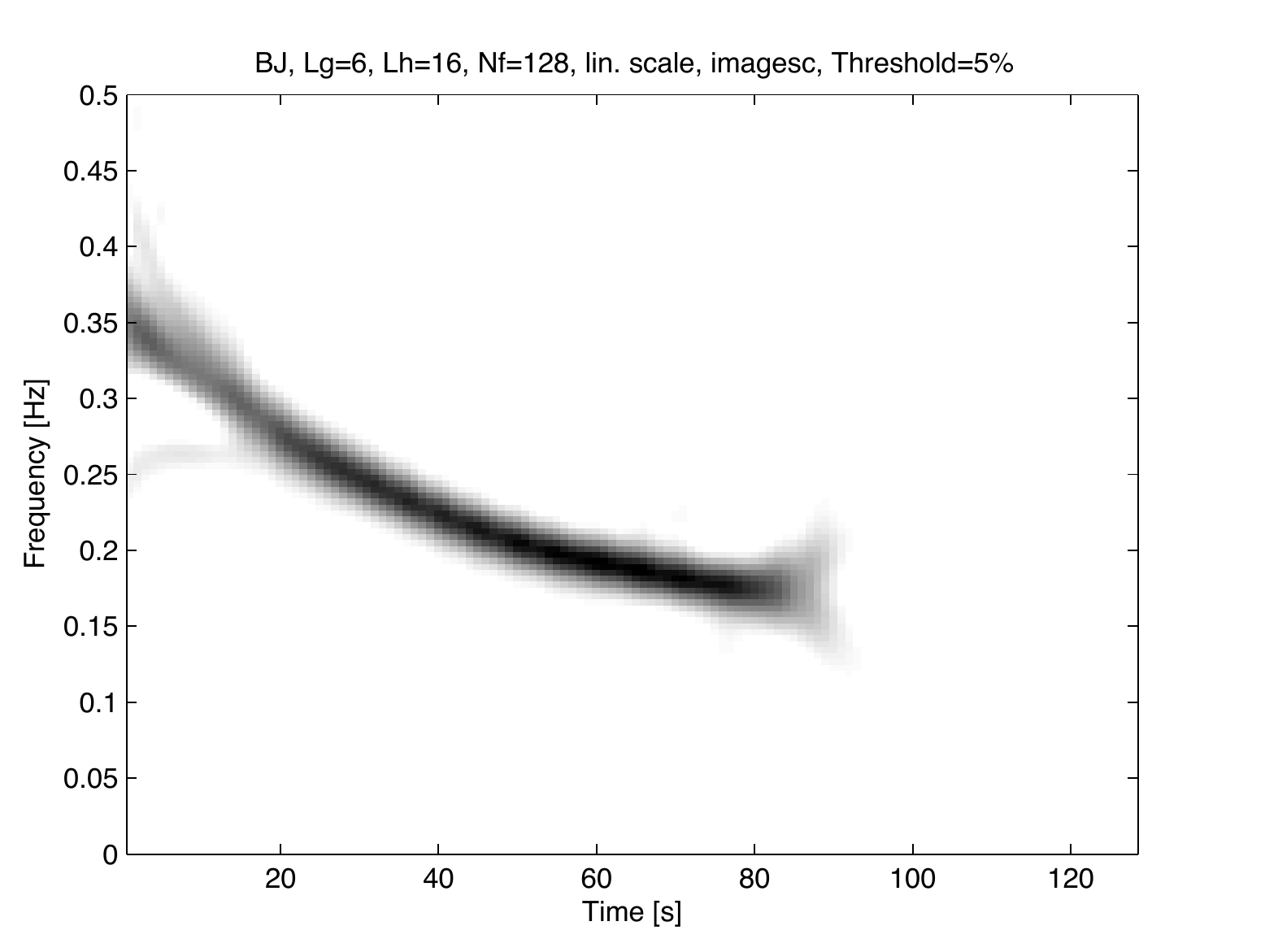}
\caption{Wigner and Born-Jordan distribution of a bat sonar signal.}
\label{figura3}
\end{figure}
A detailed and very clear mathematical discussion of the effect in the point a) above was provided in \cite{bogetal} in the model case of signals given by the sum of two Gabor atoms, so that there is only one interference term. Here we address to the above phenomena from a more general perspective; in fact the results below hold for any signal $f\in \cS'(\rd)$ (and in particular in any dimension).\par
\subsection{Description of the results}
In the engineering literature, the above principle is usually justified as the outcome of the convolution in \eqref{bj}, and the specific form of the kernel $\Theta$. While this is certainly true, a little reflection shows that the issue is of a more subtle nature. The point is that the function $\Theta_\sigma$ in \eqref{bj} enjoys a very limited smoothness, which turns out to be difficult to quantify in terms of decay of its (symplectic) Fourier transform $\Theta(z)$. In fact, the function $\Theta(z)=\Theta(\zeta_1,\zeta_2)$ is constant on the hypersurface $\zeta_1\zeta_2={\rm const.}$, so that it does not decay pointwise at infinity, nor in $L^p$ mean, because $\Theta\not\in L^p$ for $p<\infty$. Moreover one can check (see \cite{cgnb}) that
\[
\Theta_\sigma\not\in L^\infty_{loc}(\rdd)\ \ {\rm and}\ \ \Theta_\sigma\not\in L^p(\rdd)\ {\rm for\ any}\ 1\leq p\leq\infty,
\]
so that the above phenomena are admittedly not completely evident a priori. Also, we are not aware of a {\it quantitative} explanation of the above effects except for examples or model signals. \par
In the present note we provide a rigorous study in terms of suitable function spaces. As we will see, the right spaces to quantify such a mild regularity and decay seem to be the modulation space $M^{p,q}$, $1\leq p,q\leq\infty$, which were introduced 
by H. Feichtinger 
in the 80's \cite{F1} and are nowadays widely used in Time-frequency Analysis as a standard tool to measure the time-frequency concentration of a signal \cite{Birkbis,grochenig}. We recall their precise definition in Section 2 below, but for heuristic purposes in this introduction it is sufficient to think of a function $f$ in $M^{p,q}$ as a function which locally is in the Fourier-Lebesgue space $\Fur L^q$ --the space of functions whose Fourier transform is in $L^q$-- and decays on average, at infinity, as a function in $L^p$. We have $M^{p_1,q_1}\subseteq M^{p_2,q_2}$ if $p_1\leq p_1$, $q_1\leq q_2$, and for $p=q=2$ we have $M^{2,2}=L^2$. \par
Now, it is intuitively clear that the Born-Jordan distribution of a signal is certainly not rougher than the corresponding Wigner distribution; this is in fact our first basic result.
\begin{theorem}\label{teo2-zero}
Let $f\in\cS'(\rd)$ be a signal, with $Wf\in M^{p,q}(\rdd)$ for some $1\leq p,q\leq\infty$. Then $Qf\in M^{p,q}(\rdd)$.
\end{theorem}
To capture the above smoothing effect we need a microlocalized version of modulation spaces. In fact we can talk meaningfully of functions $f$ in $M^{p,q}$ at some point and in some direction, and consequently distinguish different directions in the time-frequency plane.  Since the matter is now local, the above exponent $p$ does no longer play any role, and the relevant tool turns out to be the concept of Fourier-Lebesgue wave-front set. \par
The notion of $C^\infty$ wave-front set of a distribution is nowadays a standard tool in the study of the singularities of solutions to partial differential operators. The basic idea is to detect the location and orientation of the singularities of a distribution $f$ by looking at which directions the Fourier transform of $\varphi f$ fails to decay rapidly, where $\varphi$ is a cut-off function supported sufficiently near any given point $x_0$. This test is implicitly used in  several techniques for edge detection, where often the Fourier transform is replaced by other transforms, see e.g.\ \cite{kuty} and the references therein. Actually here we are going to use a more refined notion of wave-front set introduced in \cite{ptt1,ptt2,ptt3} and involving the Fourier-Lebesgue spaces $\Fur L^q_s(\rd)$, $s\in\R$, $1\leq q\leq\infty$. The definition goes as follows (see also Section 2 below). \par
First of all we recall that the norm in the space $\Fur L^q_s(\rd)$ is given by
\begin{equation}\label{eq4-0}
\|f\|_{\Fur L^q_s(\rd)}=\|\widehat{f}(\omega) \langle \omega\rangle^{s}\|_{L^q(\rd)},
 \end{equation}
where as usual $\langle \omega\rangle=(1+|\omega|^2)^{1/2}$. Inspired by this definition, given a distribution $f\in\cD'(\rd)$ we define its wave-front set $WF_{\Fur L^q_s} (f)\subset\rd\times(\rd\setminus\{0\})$, as the set of points $({x}_0,{\omega}_0)\in \rd\times\rd$, ${\omega}_0\not=0$, where the following  condition is {\it not} satisfied: for some cut-off function $\varphi\in C^\infty_c(\rd)$ with $\varphi({x}_0)\not=0$ and some open conic neighborhood $\Gamma\subset\rd\setminus\{0\}$ of ${\omega}_0$ we have
 \begin{equation}\label{eq4}
\|\Fur [\varphi f](\omega) \langle \omega\rangle^{s}\|_{L^q(\Gamma)}<\infty.
 \end{equation}
We notice that $WF_{\Fur L^2_s} (f)=WF_{H^s}(f)$ is the usual $H^s$ wave-front set (see e.g.\ \cite[Chapter XIII]{hormander2} and Section 2 below).
 Roughly speaking, $({x}_0,{\omega}_0)\not\in WF_{\Fur L^q_s}(f)$ means that $f$ has regularity $\Fur L^q_s$ at ${x}_0$ and in the direction ${\omega}_0$. \par
 Now, we study the $\Fur L^q_s$ wave-front set of the Born-Jordan distribution of a given signal. Let us observe, en passant, that the $C^\infty$ wave-front set of the Wigner distribution also appeared in \cite{bo2} as a tool to study instantaneous frequencies.
 \begin{theorem}\label{mainteo}
 Let $f\in\cS'(\rd)$ be a signal, with $Wf\in M^{\infty,q}(\rdd)$ for some $1\leq q\leq\infty$. Let $({z},{\zeta})\in \rdd\times\rdd$, with ${\zeta}=({\zeta}_1,{\zeta}_2)$ satisfying ${\zeta}_1\cdot{\zeta}_2\not=0$. Then
 \[
 ({z},{\zeta})\not \in WF_{\Fur L^q_2}(Qf).
 \]
 \end{theorem}
Roughly speaking, if the Wigner distribution $Wf$ has local regularity $\Fur L^q$ and some control at infinity, then $Qf$ is smoother, possessing $s=2$ additional derivatives, at least in the directions ${\zeta}=({\zeta}_1,{\zeta}_2)$ satisfying ${\zeta}_1\cdot{\zeta}_2\not=0$. In dimension $d=1$ this condition reduces to ${\zeta}_1\not=0$ and ${\zeta}_2\not=0$. Hence this result explains the above smoothing phenomenon, which involves all the directions except those of the coordinates axes. In particular, since the interferences of two components which do not share the same time or frequency localization are arranged along an oblique line, they came out substantially reduced.

 \par
In particular we have the following result.
\begin{corollary}\label{cor}
Let $f\in L^2(\rd)$, so that $Wf\in L^2(\rdd)$. Let $(z,\zeta)$ be as in the statement of Theorem \ref{mainteo}. Then $({z},{\zeta})\not \in WF_{H^2}(Qf)$, i.e. $Qf$ has regularity $H^2$ at $z$ and in the direction $\zeta$.
\end{corollary}
Theorem \ref{mainteo} by itself does not preclude the possibility that some subtle smoothing effect could occur even in the directions  ${\zeta}=({\zeta}_1,{\zeta}_2)$ satisfying ${\zeta}_1\cdot{\zeta}_2=0$. We however will show that this is not the case, at least if smoothness is measured in the scale of modulation spaces.
\begin{theorem}\label{teo3}
Suppose that for some $1\leq p,q_1,q_2\leq \infty$ and $C>0$ we have 
\begin{equation}\label{test}
\|Qf\|_{M^{p,q_1}}\leq C\|W f\|_{M^{p,q_2}}
\end{equation}
for every $f\in\cS(\rd)$. Then $q_1\geq q_2$.
\end{theorem}
In other terms, for a general signal, the Born-Jordan distribution is not everywhere smoother than the corresponding Wigner distribution. Needless to say, the problems arise in the directions $\zeta=(\zeta_1,\zeta_2)$ such that  $\zeta_1\cdot\zeta_2=0$. \par
As a final remark, besides Born-Jordan, the ``correct marginals" property is trivially satisfied by any $\Theta$ which is a function of the product of its variables: $\Theta(\zeta_1,\zeta_2) = \Phi(\zeta_1\cdot\zeta_2)$, with $\Phi(0) = 1$ and, for a purpose of interference reduction, any decaying $\Phi$ is a candidate. From a practical point of view, the most popular choice has never really been the {\rm sinc} of Born-Jordan, but rather the Gaussian leading to the so-called ``Choi-Williams distribution''. Since the Gaussian enjoys much more regularity properties than the {\rm sinc} function, better results are expected. More generally, any $\Phi$ as above is certainly worth investigating in this connection.

 \par\medskip
The paper is organized as follows. In Section 2 we collect some preliminary results from Time-frequency Analysis; in particular we review the definition and basic properties of modulation and Wiener amalgam spaces. In Section 3 we study the Born-Jordan kernel $\Theta$ from the point of view of the Time-frequency Analysis. Section 4 is devoted to the proof of Theorem \ref{mainteo}, as well as a related global variant, which also implies Theorem \ref{teo2-zero} above. In Section 5 we prove the negative result stated in Theorem \ref{teo3}. At the end of the paper we report on some technical notes on the toolbox used to produce the above figures.

\section{Preliminaries}
\subsection{Notation}
We denote by $x\omega=x\cdot\omega=x_1\omega_1+\ldots +x_d\omega_d$ the scalar product in $\rd$. The notation $\langle \cdot,\cdot\rangle$ stands for the inner product in $L^2(\rd)$, or also for the duality pairing between Schwartz functions and temperate distributions (antilinear in the second argument). For functions $f,g$, we use the notations $f\lesssim g$ to denote $f(x)\leq C g(x)$ for every $x$ and some constant $C$, and similarly for $\gtrsim$. We write $f\asymp g$ for $f\lesssim g$ and $ f\gtrsim g$.\par

 We denote by $\sigma$ the standard symplectic
form on the phase space $\mathbb{R}^{2d}\equiv\mathbb{R}^{d}\times
\mathbb{R}^{d}$; the phase space variable is denoted $z=(x,\omega)$ and the dual variable by $\zeta=(\zeta_1,\zeta_2)$. By definition
$\sigma(z,\zeta)=Jz\cdot \zeta=\omega\cdot \zeta_1-x\cdot \zeta_2$, where
\[J=%
\begin{pmatrix}
0_{d\times d} & I_{d\times d}\\
-I_{d\times d} & 0_{d\times d}%
\end{pmatrix}.
\]

The Fourier transform of a function $f(x)$ in $\rn$ is normalized as
\[
\Fur f(\omega)=\widehat{f}(\omega)= \int_{\rn} e^{-2\pi i x\omega} f(x)\, dx,
\]
and the symplectic Fourier transform of a function $F(z)$ in phase space $\rnn$ is
\[
\Fur_\sigma F(\zeta)=\int_{\rdd} e^{-2\pi {i}\sigma(\zeta,z)} F(z)\, dz.
\]
We observe that the symplectic Fourier transform is an involution, i.e. $\Fur_\sigma(\Fur_\sigma F)=F$, and moreover $\Fur_\sigma F(\zeta)= \Fur F(J \zeta)$. We will also use  the important relation
\begin{equation}\label{eq10}
\Fur_\sigma[F\ast G]=\Fur_\sigma F\, \Fur_\sigma G.
\end{equation}
For $s\in\mathbb{R}$ the $L^2$-based Sobolev space $H^s(\rd)$ is constituted by the distributions $f\in\cS'(\rd)$ such that
\begin{equation}\label{normhs}
\|f\|_{H^s}:=\| \widehat{f}(\omega) \langle\omega\rangle^s \|_{L^2}<\infty.
\end{equation}

\subsection{Wigner distribution and ambiguity function \cite{Birkbis,grochenig}}
We already defined in Introduction, see \eqref{wigner}, the Wigner distribution $Wf$ of a signal $f\in \cS'(\rd)$. In general, we have $Wf\in\cS'(\rdd)$. When $f\in L^2(\rd)$ we have $Wf\in L^2(\rdd)$ and in fact it turns out 
\begin{equation}\label{wigner2}
\|Wf\|_{L^2(\rdd)}=\|f\|_{L^2(\rd)}^2.
\end{equation}
 In the sequel we will encounter several times the symplectic Fourier transform of $W f$, which is known as the (radar) {\it ambiguity function} $Af$. We have the formula
\begin{equation}\label{ambiguity}
Af(\zeta_1,\zeta_2)=\Fur_\sigma Wf (\zeta_1,\zeta_2)= \int_{\rd} f\big(y+\frac{1}{2}\zeta_1\big) \overline{f\big(y-\frac{1}{2}\zeta_1\big)}e^{-2\pi i\zeta_2 y}\, dy.
\end{equation}
We refer to \cite[Chapter 9]{Birkbis} and in particular to \cite[Proposition 175]{Birkbis} for more details.

\subsection{Modulation spaces and Wiener amalgam spaces \cite{Birkbis,feichtinger80,feichtinger83,feichtinger90,grochenig}}
Modulation spaces and Wiener amalgam spaces are used in Time-frequency Analysis to measure the time-frequency concentration of a signal. Their construction relies on the notion of short-time (or windowed) Fourier transform, which we are going to recall.\par
For $x,\omega\in \rd$ we define the translation and modulation operators
\[
T_x f(y)=f(y-x),\quad M_\omega f(y)=e^{2\pi i y\omega}f(y),
\]
and the time-frequency shifts
\[
\pi(z) f(y)=M_\omega T_x f(y)= e^{2\pi i y\omega}f(y-x),\quad z=(x,\omega).
\]
Fix a Schwartz function $g\in\cS(\rd)\setminus\{0\}$ (the so-called
{\it window}). Let $f\in\cS'(\rd)$. We define the short-time Fourier transform of $f$ as
\begin{equation}\label{STFTdef}
V_gf(z)=\langle f,\pi(z)g\rangle=\Fur [fT_x g](\omega)=\int_{\Ren}
 f(y)\, {\overline {g(y-x)}} \, e^{-2\pi iy \o }\,dy
\end{equation}
for $z=(x,\omega)\in\rd\times\rd$.\par
 Let now $s\in\mathbb{R}$, $1\leq p,q\leq
\infty$. The {\it
  modulation space} $M^{p,q}_s(\Ren)$ consists of all tempered
distributions $f\in \cS' (\rd) $ such that
\begin{equation}\label{defmod}
\|f\|_{M^{p,q}_s}:=\left(\int_{\Ren}
  \left(\int_{\Ren}|V_gf(x,\omega) |^p\langle \omega\rangle^{sp}\,
    dx\right)^{q/p}d\o\right)^{1/q}<\infty  \,
\end{equation}
(with obvious changes for $p=\infty$ or $q=\infty$).
 When $s=0$ we write $M^{p,q}(\rd)$ in place of $M^{p,q}_0(\rd)$. The spaces $M^{p,q}_s(\rd)$ are Banach spaces, and
 every nonzero $g\in \mathcal{S}(\rd)$ yields an equivalent norm in
 \eqref{defmod}.\par
 Modulation spaces generalize and include as special cases several function spaces arising in Harmonic Analysis. In particular for $p=q=2$ we have
 $$M^{2,2}_s(\rd)=H^s(\rd),$$
 whereas $M^{1,1}(\rd)$ coincides with the Segal algebra, and $M^{\infty,1}(\rd)$ is the so-called Sj\"ostrand class.\par
 As already observed in Introduction, in the notation $M^{p,q}_s$ the exponent $p$ is a measure of decay at infinity (on average) in the scale of spaces $\ell^p$, whereas the exponent $q$ is a measure of smoothness in the scale $\Fur L^q$. The index $s$ is a further regularity index, completely analogous to that appearing in the Sobolev spaces $H^s(\rd)$. \par
 
The {\it Wiener amalgam space} $W(\Fur L^p,L^q)(\rd)$ is given by the distributions $f\in\cS'(\rd)$ such that
\[
\|f\|_{W(\Fur L^p,L^q)(\rd)}:=\left(\int_{\Ren}
  \left(\int_{\Ren}|V_gf(x,\omega)|^p\,
    d\o\right)^{q/p}d x\right)^{1/q}<\infty  \,
\]
(with obvious changes for $p=\infty$ or $q=\infty$). 
Using Parseval identity in \eqref{STFTdef}, we can write the so-called fundamental identity of \tfa\, 
\[
V_g f(x,\o)= e^{-2\pi i x\o}V_{\hat g} \hat f(\o,-x),
\] hence $$|V_g f(x,\o)|=|V_{\hat g} \hat f(\o,-x)| = |\mathcal F (\hat f \, T_\o \overline{\hat g}) (-x)|$$  so that
 $$
\| f \|_{{M}^{p,q}} = \left( \int_{\rd} \| \hat f \ T_{\o} \overline{\hat g} \|_{\cF L^p}^q \ d \o \right)^{1/q}
= \| \hat f \|_{W(\cF L^p,L^q)}.
$$
This means that the Wiener amalgam spaces are simply the image under \ft\, of modulation spaces: $\cF ({M}^{p,q})=W(\cF L^p,L^q)$. \par
We will often use the following product property of Wiener amalgam spaces (\cite[Theorem 1 (v)]{feichtinger80}):
\begin{equation}\label{product}
\textit{If $f\in W(\Fur L^1,L^\infty)$ and $g\in W(\Fur L^p,L^q)$ then $fg\in W(\Fur L^p,L^q)$}.
\end{equation}
\subsection{Dilation properties of modulation and Wiener amalgam spaces}
We recall here a few dilation properties (cf.\ \cite[Lemma
3.2]{sugitomita2} and its generalization in \cite[Corollary 3.2]{CNJFA2008}).\par
\begin{proposition}\label{c1}
Let $1\leq p,q\leq\infty$ and  $A\in
GL(d,\R)$. Then, for every $f\in
\fpq(\rd)$,
\begin{equation}\label{dilAW0}
\|f(A\,\cdot)\|_{\fpq}\leq C |\det
A|^{(1/p-1/q-1)}(\det(I+A^*
A))^{1/2}\|f\|_{\fpq}.
\end{equation}
In particular, for $A=\lambda I$, $\lambda>0$,
\begin{equation}\label{dillambda}
\|f(A\,\cdot)\|_{\fpq}\leq C \lambda^{d\left(\frac1p-\frac1q-1\right)}(\lambda^2+1)^{d/2} \|f\|_{\fpq}.
\end{equation}
\end{proposition}
As a byproduct of this result, we observe that the spaces $W(\Fur L^p,L^q)$ are invariant by linear changes of variables.\par
The following result was first proved in \cite[Lemma 1.8]{toft} (see also \cite[Lemma 3.2]{CNJFA2008}).
\begin{lemma}\label{lemma5.2-zero}
Let $\varphi(x)=e^{-\pi|x|^2}$ and $\lambda>0$. Then we have the following dilation properties:
\[
\|\varphi(\lambda\,\cdot)\|_{M^{p,q}}\asymp \lambda^{-d/q'}\quad {\rm as}\ \lambda\to+\infty
\]
\[
\|\varphi(\lambda\,\cdot)\|_{M^{p,q}}\asymp \lambda^{-d/p}\quad {\rm as}\ \lambda\to0.
\]
\end{lemma}
Finally we will need the following result.
\begin{lemma}\label{lemma5.2}
Let $\psi\in C^\infty_c(\rd)\setminus\{0\}$ and $\lambda>0$. Then we have the following dilation properties:
\begin{equation}\label{eqa0}
\|\psi(\lambda\,\cdot)\|_{W(\Fur L^p,L^q)}\asymp \lambda^{-d/p'}\quad {\rm as}\ \lambda\to+\infty
\end{equation}
\begin{equation}\label{eqa1}
\|\psi(\lambda\,\cdot)\|_{W(\Fur L^p,L^q)}\asymp \lambda^{-d/q}\quad {\rm as}\ \lambda\to0.
\end{equation}
\end{lemma}
\begin{proof}
The formula \eqref{eqa0} follows by observing that for, say, $\lambda\geq 1$, the functions $\psi(\lambda\,\cdot)$ are supported in a fixed compact set, so that their $M^{p,q}$ norm is equivalent to their norm in $\Fur L^p$, which is easily estimated.\par
Let us now prove \eqref{eqa1}. 
Observe, first of all, that for $f\in \cS(\rd)$, $g\in C^\infty_c(\rd)$, $q\geq 1$ we have 
\[
\|f T_x g\|_{\Fur L^1}\lesssim \sum_{|\alpha|\leq d+1}\|\partial^\alpha (fT_x g)\|_{L^1}\lesssim \sum_{|\alpha|\leq d+1}\|\partial^\alpha (fT_x g)\|_{L^q},
\]
so that 
\[
\|f\|_{W(\Fur L^p,L^q)}\lesssim \|f\|_{W(\Fur L^1,L^q)}\lesssim\sum_{|\alpha|\leq d+1}\|\partial^\alpha f\|_{L^q}.
\]
Applying this formula with $f=\psi (\lambda x)$, $0<\lambda\leq 1$, we obtain 
\[
\|\psi(\lambda\,\cdot)\|_{W(\Fur L^p,L^q)}\lesssim \lambda^{-d/q}.
\]
To obtain the lower bound we observe that
\begin{align*}
\lambda^{-d}\|\psi\|_{L^2}^2&=\|\psi(\lambda\,\cdot)\|_{L^2}^2\\
&\lesssim\|\psi(\lambda\,\cdot)\|_{W(\Fur L^{p'},L^{q'})}\|\psi(\lambda\,\cdot)\|_{W(\Fur L^{p},L^{q})}\\
&\lesssim \lambda^{-d/q'} \|\psi(\lambda\,\cdot)\|_{W(\Fur L^{p},L^{q})},
\end{align*}
which implies
\[
\|\psi(\lambda\,\cdot)\|_{W(\Fur L^p,L^q)}\gtrsim \lambda^{-d/q}.
\]
\end{proof}

\subsection{Wave-front set for Fourier-Lebesgue spaces \cite{hormander2,ptt1}} We have briefly discussed in Introduction the idea underlying the concept of $C^\infty$ wave-front set of a distribution. Among the several refinements which have been proposed, the notion of $H^s$ wave-front set allows one to quantify the regularity of a function or distribution in the Sobolev scale, at any given point and direction. This is done by microlocalizing the definition of the $H^s$ norm in \eqref{normhs} as follows (cf.\ \cite[Chapter XIII]{hormander2}). \par
Given a distribution $f\in\cD'(\rd)$ we define its wave-front set $WF_{H^s} (f)\subset\rd\times(\rd\setminus\{0\})$, as the set of points $({x}_0,{\omega}_0)\in \rd\times\rd$, ${\omega}_0\not=0$, where the following  condition is {\it not} satisfied: for some cut-off function $\varphi\in C^\infty_c(\rd)$ with $\varphi({x}_0)\not=0$ and some open conic neighborhood of $\Gamma\subset\rd\setminus\{0\}$ of ${\omega}_0$ we have
 \[
\|\Fur [\varphi f](\omega) \langle \omega\rangle^{s}\|_{L^2(\Gamma)}<\infty.
 \]
 More in general one can start from the Fourier-Lebesgue spaces $\Fur L^q_s(\rd)$, $s\in\R$, $1\leq q\leq\infty$, which is the space of distributions $f\in\cS'(\rd)$ such that the norm in \eqref{eq4-0} is finite.  Arguing exactly as above (with the space $L^2$ replaced by $L^q$) one then arrives in a natural way to a corresponding notion of wave-front set $WF_{\Fur L^q_s}(f)$ as we anticipated in Introduction (see \eqref{eq4}). \par
We now recall from \cite{ptt1} some basic results about the action of partial differential operators on such a wave-front set. We consider the simplified case of constant coefficient operators, since this will suffice for our purposes.\par
Consider a constant coefficient linear partial differential operator
\[
P=\sum_{|\alpha|\leq m} c_\alpha \partial^\alpha
\]
where $c_\alpha\in\bC$. Then it is elementary to see that, for $1\leq q\leq\infty$, $s\in\R$, $f\in \cD'(\rd)$,
\[
WF_{\Fur L^q_s}(Pf) \subset  WF_{\Fur L^q_{s+m}}(f).
\]
Consider now the inverse inclusion. We say that $\zeta\in\rd$, $\zeta\not=0$, is non characteristic for the operator $P$ if
\[
\sum_{|\alpha|=m} c_\alpha \zeta^\alpha \not=0.
\]
This means that $P$ is elliptic in the direction $\zeta$. We have then the following result, which is a microlocal version  of the classical regularity result of elliptic operators (see \cite[Corollary 1 (2)]{ptt1}).
\begin{proposition}\label{pro3}
Let  $1\leq q\leq\infty$, $s\in\R$ and
$f\in \cD'(\rd)$. Let $z\in\rd$ and suppose that $\zeta\in\rd\setminus\{0\}$ is non characteristic for $P$. 
 Then, if $(z,\zeta)\not\in WF_{\Fur L^q_s}(Pf)$ we have $(z,\zeta)\not\in WF_{\Fur L^q_{s+m}}(f)$.
\end{proposition}

\section{Time-frequency Analysis of the Born-Jordan kernel}
In this section we show that the Born-Jordan kernel $\Theta(\zeta)$ in \eqref{sincxp} belongs to the space $W(\Fur L^1, L^\infty)(\rdd)$. This information will be a basic ingredient for the analysis in the next sections. \par
Let us recall the distribution \ft \, of the generalized chirp below (cf.\ \cite[Appendix A, Theorem 2]{folland}):
\begin{proposition}
Let $B$ a real, invertible, symmetric $n\times n$ matrix, and let $F_B(x)=e^{\pi i x B x}$. Then the distribution \ft\, of $F_B$ is given by
\begin{equation}\label{FTchirp}
   \widehat{F_B}(\omega) =e^{\pi i \sharp (B)/4} |\det B| e^{-\pi i \omega B^{-1}\omega},
\end{equation}
where $\sharp (B)$ is the number of positive eigenvalues of $B$ minus the number of negative eigenvalues.
\end{proposition}
If we choose
\[
B=B^{-1}=%
\begin{pmatrix}
0_{d\times d} & I_{d\times d}\\
I_{d\times d} & 0_{d\times d}%
\end{pmatrix}
\]
(hence $n=2d$) then formula \eqref{FTchirp} becomes
\begin{equation}\label{FTxp}
   \cF(e^{2\pi i \zeta_1\zeta_2})(z_1,z_2)=\widehat{F_B}(z)= e^{-2\pi i z_1 z_2}.
\end{equation}
We now show that $F_B$ belongs to $W(\cF L^1,L^\infty)(\rdd)$. 
\begin{proposition}\label{pro1}
The function $F(\zeta_1,\zeta_2)= e^{ 2\pi i \zeta_1 \zeta_2}$ belongs to $W(\cF L^1,L^\infty)(\rdd)$.
\end{proposition}
\begin{proof}
Consider  the Gaussian function $g(\zeta_1,\zeta_2)=e^{-\pi \zeta_1^2} e^{-\pi \zeta_2^2}$ as window function to compute the $W(\cF L^1,L^\infty)$-norm. Then we have \[
\|F\|_{W(\cF L^1,L^\infty)(\rdd)}=\sup_{u\in\rdd}\|\cF( F T_u g)\|_{L^1(\rdd)}.
\] Let us compute $\cF( F T_u g)(z)$. For $z=(z_1,z_2)$,
\begin{align*}
 \cF( F T_u g)&(z_1,z_2)=(\cF(F)\ast M_{-u}\hat{g})(z_1,z_2)\\
 &=\int_{\rdd} e^{-2\pi i (z_1- y_1)\cdot(z_2-y_2)} e^{-2\pi i (u_1,u_2)\cdot (y_1,y_2)} e^{-\pi y_1^2} e^{-\pi y_2^2}  \,dy_1 dy_2\\
 &=e^{-2\pi i z_1 z_2} \intrdd e^{-2\pi i y_1 y_2 +2\pi i (z_2 y_1+z_1 y_2)-2\pi i (u_1 y_1 + u_2 y_2)} e^{-\pi y_1^2} e^{-\pi y_2^2}\, dy_1 dy_2\\
 &= e^{-2\pi i z_1 z_2}\intrd e^{2\pi i (z_1 y_2-u_2 y_2)} e^{-\pi y_2^2}\left( \intrd e^{-2\pi i y_1\cdot ( y_2-z_2 +u_1)} e^{-\pi y_1^2}\, dy_1\right) \, dy_2\\
 &=e^{-2\pi i z_1 z_2} \intrd e^{-2\pi i y_2\cdot (u_2-z_1)} e^{-\pi y_2^2} e^{-\pi(y_2-z_2+u_1)^2}\,dy_2\\
 &=e^{-2\pi i z_1 z_2} e^{-\pi (u_1-z_2)^2}\intrd e^{-2\pi i y_2\cdot (u_2-z_1)} e^{-2\pi (y_2^2+(u_1-z_2)\cdot y_2)}\, dy_2\\
 &=e^{-2\pi i z_1 z_2} e^{-\pi (u_1-z_2)^2+\frac\pi 2 (u_1-z_2)^2}\intrd e^{-2\pi i y_2 \cdot (u_2-z_1)} e^{-2\pi (y_2+\frac{u_1-z_2}2)^2}\,dy_2\\
 &=e^{-2\pi i z_1 z_2} e^{-\frac\pi 2 (u_1-z_2)^2}\cF\left(T_{-\frac{u_1-z_2}2}e^{-2\pi |\cdot|^2}\right)(u_2-z_1)\\
 &=2^{-d/2}e^{-2\pi i z_1 z_2} e^{-\frac\pi2 (u_1-z_2)^2} e^{-\frac\pi2 (u_2-z_1)^2} e^{\pi i (u_1-z_2)\cdot (u_2-z_1)}\\
 &=2^{-d/2} e^{\pi i (u_1 u_2-z_1 z_2 -u_1z_1-u_2z_2) } e^{-\frac\pi2 (u_1-z_2)^2} e^{-\frac\pi2 (u_2-z_1)^2}.
\end{align*}
Hence \[
\|\cF( F T_u g)\|_{L^1}=2^{-d/2}  \|e^{-\frac\pi 2 |\cdot|^2}\|_{L^1}=C,
\]
for a constant $C$ independent of the variable $u$ and, consequently,
$$\sup_{u\in\rdd}\|\cF( F T_u g)\|_{L^1}<\infty$$ as desired.
\end{proof}
\begin{remark}\label{osservazione}\rm
Since $W(\cF L^1,L^\infty)(\rdd)$ can be characterized as the space of pointwise multipliers on the Segal algebra $W(\cF L^1,L^1)(\rdd)$ \cite[Corollary 3.2.10]{feizim}, the result in Proposition \ref{pro1} could also be deduced from general results about the action of second order characters on the Segal algebra, cf.\ \cite{fei0,reiter}. However, the proof above is clearly more direct and elementary.\par
We also observe that by the restriction property of the Segal algebra, one also obtains, as a consequence, that the chirp function $e^{\pi i |x|^2}$ belongs to $W(\Fur L^1,L^\infty)(\rd)$, which is well known.  
\end{remark}
\begin{corollary}\label{cor1} For $\zeta=(\zeta_1,\zeta_2)$, consider the  function $F_J(\zeta)=F( J \zeta)= e^{- 2\pi i \zeta_1 \zeta_2}$.
 Then  $F_J\in W(\cF L^1,L^\infty)(\rdd)$.
\end{corollary}
\begin{proof} The result immediately follows by Proposition \ref{pro1} and  by the dilation properties for Wiener amalgam spaces \eqref{dilAW0}.
\end{proof}

For $\zeta=(\zeta_1,\zeta_2)\in\rd\times\rd$, we have defined the Born-Jordan kernel $\Theta(\zeta)$ in \eqref{sincxp}.

\begin{proposition}\label{pro2}
The function $\Theta$ in \eqref{sincxp} belongs to $W(\cF L^1,L^\infty)(\rdd)$.
\end{proposition}
\begin{proof}
It is well known that
 $${\rm sinc}(\zeta_1\zeta_2)=\int_{-1/2}^{1/2} e^{-2\pi i \zeta_1\zeta_2 t}\,d t=\int_{0}^{1/2} e^{2\pi i \zeta_1\zeta_2 t}\,d t+ \int_{0}^{1/2} e^{-2\pi i \zeta_1\zeta_2 t}\,d t.$$
Now, we have proved in Proposition \ref{pro1} that the function $e^{2\pi i \zeta_1\zeta_2 }$ belongs to the space $W(\cF L^1,L^\infty)(\rdd)$, and the same holds for $e^{-2\pi i \zeta_1\zeta_2}$, by Corollary \ref{cor1}. Using the dilation relations for Wiener amalgam spaces \eqref{dillambda} for $\lambda=\sqrt{t}$, $0<t<1/2$, $p=1$, $q=\infty$, we obtain
$$\|e^{\pm 2\pi i \zeta_1\zeta_2 t}\|_{W(\cF L^1,L^\infty)}\leq C \|e^{\pm 2\pi i \zeta_1\zeta_2 }\|_{W(\cF L^1,L^\infty)}$$
so that
$$\|\Theta\|_{W(\cF L^1,L^\infty)}\leq C \int_{0}^{1/2}\|e^{2\pi i \zeta_1\zeta_2}\|_{W(\cF L^1,L^\infty)} + \|e^{-2\pi i \zeta_1\zeta_2}\|_{W(\cF L^1,L^\infty)}\,d t<\infty$$
and we obtain the claim.
\end{proof}
\section{Smoothness of the Born-Jordan distribution}
In the present section we study the smoothness of the Born-Jordan distribution of a signal, in comparison with the corresponding Wigner distribution. In particular we prove Theorem \ref{mainteo}.\par
 We begin with the following global result, which in particular implies Theorem \ref{teo2-zero}.
\begin{theorem}\label{teo2}
Let $f\in\cS'(\rd)$ be a signal, with $Wf\in M^{p,q}(\rdd)$ for some $1\leq p,q\leq\infty$. Then \[
Qf\in M^{p,q}(\rdd)
\]
 and moreover
\begin{equation}\label{eq1}
\nabla_{x}\cdot \nabla_{\omega} Qf\in M^{p,q}(\rdd).
\end{equation}
\end{theorem}
Here we used the notation \[
\nabla_{x}\cdot \nabla_{\omega}:=\sum_{j=1}^d \frac{\partial^2}{\partial {x_j}\partial {\omega_j}}.
\]
\begin{proof}
Let us first prove that $Qf\in M^{p,q}(\rdd)$. Taking the symplectic Fourier transform in \eqref{bj} we are reduced to prove that
\[
\Theta \Fur_\sigma(Wf)\in W(\Fur L^p,L^q)
\]
(as already observed in \eqref{ambiguity}, $\Fur_\sigma(Wf)=Af$ is the ambiguity function of $f$). \par
Now, this follows from the product property \eqref{product}: in fact by Proposition \ref{pro2} we have $\Theta \in W(\Fur L^1,L^\infty)$, and by assumption $Wf\in M^{p,q}(\rdd)$ so that $\Fur(Wf)\in W(\Fur L^p,L^q)$, and therefore $\Fur_\sigma(Wf)(\zeta)=\Fur(Wf)(J\zeta)$ belongs to $W(\Fur L^p,L^q)$ as well (by Proposition \ref{c1}). \par
Let us prove \eqref{eq1}.  Taking the symplectic Fourier transform we see that it is sufficient to prove that
\[
 \zeta_1\zeta_2\, {\rm sinc}(\zeta_1\zeta_2) \Fur_\sigma Wf=\frac{1}{\pi}\sin(\pi \zeta_1\zeta_2) \Fur_\sigma Wf\in W(\Fur L^p,L^q).
\]
Now, 
\[
\sin (\pi \zeta_1\zeta_2)=\frac{e^{\pi i \zeta_1\zeta_2}-e^{-\pi i \zeta_1\zeta_2}}{2i},
\] and the claim follows  again as above using the product properties for Wiener amalgam spaces \eqref{product}, Proposition \ref{pro1} and Proposition \ref{c1} with the scaling $\lambda =1/\sqrt{2}$.

\end{proof}

We are now ready to prove Theorem \ref{mainteo}.
\begin{proof}[Proof of Theorem \ref{mainteo}]
We will apply Proposition \ref{pro3} to the second order operator $P=\nabla_{x}\cdot \nabla_{\omega}$ in $\rdd$. Observe that the non characteristic directions are given by the vectors $\zeta=(\zeta_1,\zeta_2)\in\rd\times\rd$, satisfying $\zeta_1\cdot\zeta_2\not =0$. By \eqref{eq1} (with $p=\infty$) we have  \[
WF_{\Fur L^q}(\nabla_{x}\cdot \nabla_{\omega} Qf)=\emptyset,
\]
because $\varphi F\in \Fur L^q$ if $\varphi\in C^\infty_c(\rdd)$ and $F\in M^{\infty,q}(\rdd)$ (we apply this remark with $F=\nabla_{x}\cdot \nabla_{\omega} Qf$). Hence
we have certainly \[
(z,\zeta)\not \in WF_{\Fur L^q}(\nabla_{z_1}\cdot \nabla_{z_2} Qf).\]
Since $\zeta$ is non characteristic for the operator $\nabla_{x}\cdot \nabla_{\omega}$, by Proposition \ref{pro3} we therefore have \[
(z,\zeta)\not \in WF_{\Fur L^q_2}(Qf)
\]
 for every $z\in\rdd$.
\end{proof}
\begin{proof}[Proof of Corollary \ref{cor}]
It is sufficient to apply Theorem \ref{mainteo} with $q=2$. In fact if $f\in L^2(\rd)$ we have $Wf\in L^2(\rdd)=M^{2,2}(\rd)$ (cf.\ \eqref{wigner2}) and therefore $Wf\in M^{\infty,2}(\rdd)$. Moreover, as already observed, the $\Fur L^2_2$ wave-front set coincides with the $H^2$ wave-front set.
\end{proof}
\section{Negative results (proof of Theorem \ref{teo3})}
In this section we prove Theorem \ref{teo3}.
We need the following result. 
\begin{lemma}\label{lemma5.1}
Let $\chi\in C^\infty_c(\R)$. Then the function $\chi(\zeta_1 \zeta_2)$ belongs to $W(\Fur L^1,L^\infty)(\rdd)$.
\end{lemma}
\begin{proof}
We write 
\[
\chi(\zeta_1 \zeta_2)=\int_{\R} e^{2\pi i t \zeta_1 \zeta_2} \widehat{\chi}(t)\, dt
\]
and the desired result then follows as in proof of Proposition \ref{pro2}, because $\widehat{\chi}(t)$ lies in $\cS(\R)\subset L^1(\R)$. 
\end{proof}

\begin{proof}[Proof of Theorem \ref{teo3}]
We test the estimate \eqref{test} on the Gaussian functions $f(x)=\varphi(\lambda x)$, where $\varphi(x)=e^{-\pi |x|^2}$ and $\lambda>0$ is a large parameter. \par
An explicit computation (see e.g.\ \cite[Formula (4.20)]{grochenig}) gives 
\begin{equation}\label{wignerdil}
W(\varphi(\lambda\, \cdot))(x,\omega)=2^{d/2} \lambda^{-d} \varphi(\sqrt{2}\lambda\, x)\varphi(\sqrt{2}\lambda^{-1}\, \omega).
\end{equation}
Hence we have, for every $1\leq p,q\leq\infty$,
\[
\|W(\varphi(\lambda\, \cdot))\|_{M^{p,q}}=2^{d/2} \lambda^{-d}\|  \varphi(\sqrt{2}\lambda\, \cdot) \|_{M^{p,q}}\|\varphi(\sqrt{2}\lambda^{-1}\, \cdot) \|_{M^{p,q}}.
\]
Now it follows from Lemma \ref{lemma5.2-zero} that 
\begin{equation}\label{eqa3}
\|W(\varphi(\lambda\, \cdot))\|_{M^{p,q}}\asymp \lambda^{-2d+d/q+d/p}\quad{\rm as}\ \lambda\to+\infty.
\end{equation}
Let us now estimate from below
\[
\|Q(\varphi(\lambda\, \cdot))\|_{M^{p,q}}=\|\Theta_\sigma \ast W(\varphi(\lambda\, \cdot))\|_{M^{p,q}}.
\]
By taking the symplectic Fourier transform and using Lemma \ref{lemma5.1} and the product property \eqref{product} we have
\begin{align*}
\|\Theta_\sigma \ast W(\varphi(\lambda\, \cdot))\|_{M^{p,q}}&\asymp \|\Theta \Fur_\sigma[ W(\varphi(\lambda\, \cdot))]\|_{W(\Fur L^p,L^q)}\\
&\gtrsim \|\Theta(\zeta_1,\zeta_2) \chi(\zeta_1\zeta_2)\Fur_\sigma[ W(\varphi(\lambda\, \cdot))]\|_{W(\Fur L^p,L^q)}
\end{align*} 
for any $\chi\in C^\infty_c(\R)$. We now choose $\chi$ supported in the interval $[-1/4,1/4]$ and $=1$ in the interval $[-1/8,1/8]$ (the latter condition will be used later), and we write
\[
\chi(\zeta_1 \zeta_2)=\chi(\zeta_1 \zeta_2) \Theta(\zeta_1,\zeta_2)\Theta^{-1}(\zeta_1,\zeta_2)\tilde{\chi}(\zeta_1 \zeta_2),
\]
 with $\tilde{\chi}\in C^\infty_c(\R)$ supported in $[-1/2,1/2]$ and $\tilde{\chi}=1$ on $[-1/4,1/4]$, therefore on the support of $\chi$. Since by Lemma \ref{lemma5.1} the function $\Theta^{-1}(\zeta_1,\zeta_2)\tilde{\chi}(\zeta_1 \zeta_2)$ belongs to $W(\Fur L^1,L^\infty)$, again by the product property the last expression is estimated from below as
\[
\gtrsim \| \chi(\zeta_1\zeta_2)\Fur_\sigma[ W(\varphi(\lambda\, \cdot))]\|_{W(\Fur L^p,L^q)}.
\]
We now consider a function $\psi\in C^\infty_c(\rd)\setminus\{0\}$, supported where $|\zeta_1|\leq 1/4$. Using the Cauchy-Schwarz inequality in the form
\[
|\zeta_1 \zeta_2|\leq\frac{1}{2}(|\lambda\zeta_1|^2+|\lambda^{-1}\zeta_2|^2)
\]
 we see that $\chi(\zeta_1 \zeta_2)=1$ on the support of $\psi(\lambda \zeta_1)\psi(\lambda^{-1}\zeta_2)$, for every $\lambda>0$.\par
  Then, we can write 
\[
 \psi(\lambda \zeta_1)\psi(\lambda^{-1}\zeta_2)=\chi(\zeta_1 \zeta_2)\psi(\lambda \zeta_1)\psi(\lambda^{-1}\zeta_2)
\]
and by Lemma \ref{lemma5.2} we also have 
\[
\|\psi(\lambda \zeta_1)\psi(\lambda^{-1}\zeta_2)\|_{W(\Fur L^1,L^\infty)}\lesssim 1
\]
so that we can continue the above estimate as
\[
\gtrsim \|\psi(\lambda \zeta_1)\psi(\lambda^{-1}\zeta_2)\Fur_\sigma[ W(\varphi(\lambda\, \cdot))]\|_{W(\Fur L^p,L^q)}.
\]
Using the formula \eqref{wignerdil} we have 
\[
\Fur_\sigma[ W(\varphi(\lambda\, \cdot))(\zeta_1,\zeta_2)= 2^{-d/2} \lambda^{-d}\varphi((\sqrt{2}\lambda)^{-1}\, \zeta_2)\varphi((1/\sqrt{2})\lambda\, \zeta_1), 
\]
so that we obtain 
\begin{align*}
& \|\psi(\lambda \zeta_1)\psi(\lambda^{-1}\zeta_2)\Fur_\sigma[ W(\varphi(\lambda\, \cdot))]\|_{W(\Fur L^p,L^q)}\\
&=2^{-d/2} \lambda^{-d}\|\psi(\lambda \zeta_1)\psi(\lambda^{-1}\zeta_2)  \varphi((\sqrt{2}\lambda)^{-1}\, \zeta_2)\varphi((1/\sqrt{2})\lambda\, \zeta_1) \|_{W(\Fur L^p,L^q)}\\
&=2^{-d/2} \lambda^{-d}\|\psi(\lambda \zeta_1)\varphi((1/\sqrt{2})\lambda\, \zeta_1) \|_{W(\Fur L^p,L^q)}\| \psi(\lambda^{-1}\zeta_2)\varphi((\sqrt{2}\lambda)^{-1}\, \zeta_2)\|_{{W(\Fur L^p,L^q)}}.
\end{align*}
Using Lemma \ref{lemma5.2} to estimate the last expression we deduce
\[
\|Q(\varphi(\lambda\, \cdot))\|_{M^{p,q}}\gtrsim \lambda^{-2d+d/p+d/q}\quad {\rm as}\ \lambda\to+\infty.
\]
A comparison with \eqref{eqa3} gives the desired conclusion.
\end{proof}

\section*{Technical notes}
The figures in Introduction were produced by the Time-Frequency Toolbox (TFTB), distributed under the terms of the GNU Public Licence:
\begin{center}
http://tftb.nongnu.org/
\end{center}
In particular the bat sonar signal in Figure 3 was recorded as a .mat file in that toolbox. The figures are inspired by the several examples appearing in the tutorial of the toolbox, and in the book \cite{auger}. 
\section*{Acknowledgments} The authors wish to thank H. Feichtinger for reading a preliminary version of the manuscript and for suggesting the Remark \ref{osservazione}. They also wish to thank one of the referees for pointing out the remarks that follow the Principle stated in Section 1.1 and those that follow Theorem \ref{teo3}.\par
Maurice de Gosson has been funded by the Austrian Research Fund (FWF) grant P27773-N25.


\vskip0.5truecm

\end{document}